\documentclass[12pt,a4paper,twoside]{amsart}

\usepackage{xspace, graphicx, epsfig, amssymb, amsmath, amsthm, stmaryrd}
\usepackage{epsfig}
\usepackage{color}

\usepackage{bbm}
\usepackage[all,cmtip]{xy}

\newtheorem{theorem}{Theorem}[section]
\newtheorem{corollary}[theorem]{Corollary}
\newtheorem{lemma}[theorem]{Lemma}
\newtheorem{proposition}[theorem]{Proposition}

\newtheorem{thm-dfn}[theorem]{Theorem-Definition}
\newtheorem{claim}[theorem]{Claim}
\newtheorem{clm-dfn}[theorem]{Claim-Definition}

\theoremstyle{definition}
\newtheorem{definition}[theorem]{Definition}
\newtheorem{remark}[theorem]{Remark}

\numberwithin{equation}{section}

\newcommand{\quash}[1]{}  %Anything in \quash is ignored

\newcommand{\bs}{\backslash}

\newcommand{\frakc}{{\mathfrak c}}

\newcommand{\frakg}{{\mathfrak g}}

\newcommand{\frakl}{{\mathfrak l}}

\newcommand{\fraks}{{\mathfrak s}}

\newcommand{\frakv}{{\mathfrak v}}

\newcommand{\bbB}{{\mathbb B}}
\newcommand{\bbC}{{\mathbb C}}

\newcommand{\bbG}{{\mathbb G}}

\newcommand{\bbQ}{{\mathbb Q}}
\newcommand{\bbR}{{\mathbb R}}

\newcommand{\bbZ}{{\mathbb Z}}

\newcommand{\mC}{{\mathbb C}}

\newcommand{\calB}{{\mathcal B}}

\newcommand{\calH}{{\mathcal H}}
\newcommand{\calI}{{\mathcal I}}

\newcommand{\calO}{{\mathcal O}}
\newcommand{\calP}{{\mathcal P}}

\newcommand{\calT}{{\mathcal T}}

\newcommand{\mcH}{\mathcal H}

\newcommand{\mcO}{\mathcal O}

\newcommand{\mcT}{\mathcal T}

\newcommand{\ep}{\epsilon}

\usepackage{xcolor}
\usepackage{amscd}

\newcommand{\mtrx}[4]{\left( \begin{array}{cc} #1 & #2 \\ #3 & #4 \end{array} \right)}

\newcommand{\cupdot}{\mathbin{\mathaccent\cdot\cup}}

\begin{document}

\title{On irreps of a Hecke algebra of a non-reductive group}
\author{David Kazhdan and Alexander Yom Din}

\begin{abstract}
	We study irreducible representations of the Hecke algebra of the pair \quash{$(\mathfrak{s}\mathfrak{l}_2 (F) \rtimes {\rm PGL}_2 (F) , \mathfrak{s}\mathfrak{l}_2 (\mathcal{O}) \rtimes {\rm PGL}_2 (\mathcal{O}))$} $({\rm PGL}_2 (F[\epsilon] / (\epsilon^2)) , {\rm PGL}_2 (\mathcal{O}[\epsilon] / (\epsilon^2)))$ where $F$ is a local non-Archimedean field of characteristic different than $2$ and $\mathcal{O} \subset F$ is its ring of integers. We expect to apply our analysis to the study of the spectrum of Hecke operators on the space of cuspidal functions on the space of principal ${\rm PGL}_2$-bundles on curves over rings $\mathbb{F}_q [\epsilon] / (\epsilon^2)$.
\end{abstract}

\maketitle

\tableofcontents

%%%%%%%%%%%%%%%%%%%%%%%%%%%%%%%%%%%%
%%%%%%%%%%%%%%%%%%%%%%%%%%%%%%%%%%%%

\section{Introduction}

\subsection{The content of this paper}

\subsubsection{}
Let $F$ be a non-Archimedean local field of characteristic different than $2$, $\calO$ the ring of integers of $F$, $\varpi \in \mcO$ a generator of the maximal ideal and $\kappa := \mcO/ \varpi \mcO$ the residue field. Let $G := {\rm PGL}_2 (F[\ep] / (\ep^2))$ and $K := {\rm PGL}_2 (\mcO [\ep] / (\ep^2))$. Note that $G$ is a non-reductive group, isomorphic to $\fraks\frakl_2 (F) \rtimes \rm {PGL}_2 (F)$. In this paper we study irreducible representations of the Hecke algebra $\calH$ of $K$-biinvariant compactly supported measures on $G$.

\subsubsection{}

The algebra  $\calH$ is non-commutative but it contains a commutative subalgebra $\mcT$ generated by elements $T_x$, $x \in \kappa$, defined as follows. Let $Y_x := K \mtrx{\varpi + \epsilon \tilde{x}}{0}{0}{1} K$ where $\tilde{x} \in \calO$ is a representative of $x$. Observe that the subset $Y_x \subset G$ does not depend on the choice of $\tilde{x}$. We denote by $T_x$ the characteristic measure of $Y_x$ divided by $|\kappa|$.

\subsubsection{}

Any irreducible smooth representation $V$ of $G$ defines an irreducible representation of $\calH$ on the space $V^K$ of $K$-invariants. In this article we describe the structure of the $V^K$'s as $\mathcal{T}$-modules. From this we also deduce a description of the algebra $\mathcal{T}$. All, but one, irreducible smooth representations $V$ of $G$ satisfy $\dim (V^K) < \infty$ (one can check that all these are, in addition, admissible). For many irreducible smooth representations $V$ of $G$ the action of $T_x$ on $V^K$ does not depend on $x$. We show that, for many unitary irreducible smooth representations $V$ of $G$, the eigenvalues of $T_x$ on $ V^K $ are of the form $\lambda + \bar \lambda$ where $\lambda$ is a Weil number such that $| \lambda | = |\kappa|^{1/2}$. We also see that, as $\dim V^K$ tends to infinity, the normalized eigenvalues ${\rm Re} (\lambda) / |\kappa|^{1/2}$ distribute in the inetrval $[-1, 1]$ according to the probability density $\tfrac{1}{\pi} \cdot \frac{1}{\sqrt{1-x^2}}$ (familiar from the work of Sato-Tate). For the irreducible smooth representation $V$ of $G$ for which $\dim V^K = \infty$, we find the spectral decomposition of $T_x$ acting on the $L^2$-completion of $V^K$.

\subsection{Relation to the study of \cite{BrKa}}

\subsubsection{}

Let $\kappa$ be a finite field, $C$ a smooth curve over the ring $B := \kappa [\ep] / (\ep ^2)$ and ${\rm Bun}$ the groupoid of principal ${\rm PGL}_2$-bundles on $C$. Let $H$ be the 
space of finitely supported $\mC$-valued functions 
on $ \rm{Bun}$.

\subsubsection{}

For every point $v \in C (B)$ we denote by $\calH_v$ the Hecke algebra as above, associated with the field $F_v$ of formal Laurent series at $v$ (and by $\calT_v \subset \calH_v$ the corresponding subalgebra as above). The Hecke correspondences define a representation $\rho _v : \calH_v \to {\rm End} (H)$. Let 
$R \subset {\rm End} (H) $ be the subalgebra generated by $ \cup_{v \in C(B)} \rho _v (\mcT_v)$. Then $R$ is a commutative algebra of selfadjoint operators.

\subsubsection{}

One can define a subspace $H _{\rm cusp} \subset H$ of cuspidal functions, of dimension $\sim |\kappa|^{2(3g(C)-3)}$, invariant under $R$. 

\subsubsection{}

We expect to apply our analysis of the spectrum of the local Hecke algebra $\mcH$ to the study of the spectrum of $R$ on $H_{\rm cusp}$.

\subsection{Acknowledgements}

Alexander Yom Din's research was supported by the ISRAEL SCIENCE FOUNDATION (grant No 1071/20). David Kazhdan's research was partially supported by ERC grant No 669655.

\subsection{Notations}

\subsubsection{}

The notations in the paper differ slightly from those of the introduction.

\subsubsection{}

We fix a non-Archimedean local field $F$ of characteristic different than $2$. Let $\calO$ be the ring of integers of $F$, $\varpi \in \calO$ a generator of the maximal ideal of $\calO$, $\frakv : F^{\times} \to \bbZ$ the valuation and $\kappa := \calO / \varpi \calO$. We write $q := |\kappa |$. We fix an additive character $\psi_0 : F \to {\rm U} (1)$ for which $$ \{ x \in F \ | \ \psi_0|_{x \calO} = 1 \} = \calO.$$

\subsubsection{}

Given a ring $A$, we write $A_{\epsilon} := A [\epsilon] / (\epsilon^2)$. For $a \in A_{\epsilon}$ we denote by $a_{[0]} \in A$ the element such that $a - a_{[0]} \in \epsilon A$.

\subsubsection{}

Let $\bbG := {\rm PGL}_2$ and let $\bbB \subset \bbG$ be the subgroup of upper-triangular matrices. If a ring $A$ has a trivial Picard group (for example if the reduced ring of $A$ is a unique factorization domain) then $\bbG (A) = {\rm GL}_2 (A) / A^{\times}$. We denote by $\theta : \bbG \to \bbG$ the transpose.

\subsubsection{}

Let $G := \bbG (F)$, $K := \bbG (\calO)$, $B := \bbB (F)$, $G_{\epsilon} := \bbG (F_{\epsilon})$, $K_{\epsilon} := \bbG (\calO_{\epsilon})$, $\frakg := \fraks \frakl_2 (F)$ and $\frakg_{\calO} := \fraks \frakl_2 (\calO)$. We have a semidirect decomposition $$ G_{\epsilon} = (1 + \epsilon \frakg) \rtimes G.$$

\subsubsection{}

We abuse notation and write elements in ${\rm GL}_2$, thinking of them as elements in ${\rm PGL}_2$. Given a $2\times 2$-matrix $m$, we sometimes write its entries in the following way: $$ m = \mtrx{m_{11}}{m_{12}}{m_{21}}{m_{22}}.$$

\subsubsection{}

We use the following matrix notations:
$$ w := \mtrx{0}{1}{1}{0}, \quad u_x := \mtrx{1}{x}{0}{1} \quad t_x := \mtrx{x}{0}{0}{1}.$$

\subsubsection{}

Given a locally compact topological group $H$ and an open compact subgroup $L \subset H$, we denote by $\calH_L (H)$ the Hecke algebra\footnote{We consider elements of $\calH_L (H)$ both as measures and as functions on $H$ using the Haar measure on $H$ for which the volume of $L$ is $1$.}, and we denote its multiplication by $\star$. Given a compact $L$-biinvariant subset $S \subset H$, we denote its characteristic measure by ${\rm ch}_{S} \in \calH_L (H)$.

\section{Hecke operators}

\subsection{Recollections}

\subsubsection{}\label{sssec general hecke notation}

Let $H$ be a locally compact topological group and let $L \subset H$ be an open compact subgroup. Given $h \in H$ let us denote $$ L_h := \{ \ell \in L \ | \ L h \ell  = L h  \} = L \cap h^{-1} L h$$ and $$ L^-_h := \{ \ell \in L \ | \ \ell h L = h L \} = L \cap h L h^{-1}.$$

\begin{lemma}\label{lem general hecke}
	Let $h_1 , h_2 \in H$. Then\footnote{Here and later, we sometimes run over a quotient set, implicitly meaning running over a set of representatives.} $$ L h_1 L h_2 L = \bigcup_{\ell \in L_{h_1} \bs L / L_{h_2}^-} L h_1 \ell h_2 L.$$ Given $\ell \in L$, we have $$ \left( {\rm ch}_{L h_1 L} \star {\rm ch}_{L h_2 L} \right) (h_1 \ell h_2) = \left| \left\{ (\ell_1 , \ell_2) \in L / L_{h_1}^- \times L_{h_2} \bs L \ | \ (h_2 \ell_2 h_2^{-1}) \ell^{-1} (h_1^{-1} \ell_1 h_1) \in L \right\} \right|.$$
\end{lemma}

\subsubsection{}

For the next lemma, let us be given a continuous anti-involution $\theta : H \to H$ satisfying $\theta (L) = L$. We then have an induced involutive map $\theta$ on the set of $L$-double cosets in $H$. Let us say that a $L$-biinvariant subset $S \subset H$ is \textbf{strongly $\theta$-invariant} if every $L$-double coset contained in it is $\theta$-invariant.

\begin{lemma}[Gelfand's trick]\label{lem gelfand trick}
	Let $h_1 , h_2 \in H$. If $L h_1 L$, $L h_2 L$ and $L h_1 L h_2 L$ are strongly $\theta$-invariant, then $$ {\rm ch}_{L h_1 L} \star {\rm ch}_{L h_2 L} = {\rm ch}_{L h_2 L} \star {\rm ch}_{L h_1 L}.$$
\end{lemma}

\subsection{Our elements $T_x$}

\subsubsection{}

Given $x, y \in \calO$, let us denote $$ g_x := \mtrx{\varpi + \epsilon x}{0}{0}{1},$$ $$g_{x,y} := \mtrx{\varpi + \epsilon x}{0}{0}{\varpi + \epsilon y}$$ and $$ h_{x,y} := \mtrx{-\varpi(x+y)-\epsilon xy }{\varpi}{\varpi}{\epsilon}.$$ It is immediate to check that the $K_{\epsilon}$-double cosets of these elements (and also the element $g_x g_y$ that we will consider in what follows) only depend on the images of $x$ and $y$ in $\kappa$. We therefore abuse notation in what follows, starting with elements in $\kappa$ but working with their lifts in $\calO$.

\subsubsection{}

\begin{definition}
	Given $x \in \kappa$, we denote by $T_x \in \calH_{K_{\epsilon}} (G_{\epsilon})$ the following element: $$ T_x := q^{-1} \cdot {\rm ch}_{K_{\epsilon} g_x K_{\epsilon}}.$$
\end{definition}

\subsubsection{}

\begin{lemma}\label{lem double coset computation}
Let $x,y \in \kappa$.
	\begin{enumerate}
		\item If $x = y$, we have
$$ K_{\epsilon} g_x K_{\epsilon} g_y K_{\epsilon} = K_{\epsilon} g_x g_y K_{\epsilon} \cupdot K_{\epsilon} g_{x,y} K_{\epsilon} \cupdot K_{\epsilon} h_{x,y} K_{\epsilon}.$$
		\item If $x \neq y$, we have $K_{\epsilon} g_{x,y} K_{\epsilon} = K_{\epsilon} h_{x,y} K_{\epsilon}$ and we have $$ K_{\epsilon} g_x K_{\epsilon} g_y K_{\epsilon} = K_{\epsilon} g_x g_y K_{\epsilon} \cupdot K_{\epsilon} g_{x,y} K_{\epsilon}.$$
	\end{enumerate}
\end{lemma}

\begin{corollary}\label{cor commutative}
	Given $x,y \in \kappa$, we have $$T_x \star T_y = T_y \star T_x.$$
\end{corollary}

\begin{proof}[Proof (of Corollary \ref{cor commutative}).] We use Lemma \ref{lem gelfand trick}, noticing that $$g_x, \ g_y, \ g_x g_y, \ g_{x,y}, \ h_{x,y} \in G_{\epsilon}^{\theta}.$$
\end{proof}

\begin{proof}[Proof (of Lemma \ref{lem double coset computation}).]	Let us first see that
	\begin{equation}\label{eq decomposition double classes} K_{\epsilon} g_x K_{\epsilon} g_y K_{\epsilon} = K_{\epsilon} g_x g_y K_{\epsilon} \cup K_{\epsilon} g_{x,y} K_{\epsilon} \cup K_{\epsilon} h_{x,y} K_{\epsilon}.\end{equation} It is straight-forward to see that (following the general notation of \S\ref{sssec general hecke notation}) $(K_{\epsilon})_{g_x}$ is the inverse image under the surjective map $\bbG (\calO_{\epsilon}) \to \bbG (\calO_{\epsilon} / (\varpi + \epsilon x) \calO_{\epsilon})$ of $\bbB (\calO_{\epsilon} / (\varpi + \epsilon x) \calO_{\epsilon})$, and $(K_{\epsilon})^-_{g_x} = \theta ((K_{\epsilon})_{g_x})$. Next, we compute that $$ K_{\epsilon} = (K_{\epsilon})_{g_x} (K_{\epsilon})^-_{g_y} \cup (K_{\epsilon})_{g_x} w (K_{\epsilon})^-_{g_y} \cup (K_{\epsilon})_{g_x} w u_{\epsilon} (K_{\epsilon})^-_{g_y}.$$ Indeed, let $k \in K_{\epsilon}$. If $\frakv ((k_{22})_{[0]}) = 0$ then $u_{k_{22}^{-1} k_{12}}^{-1} k \in (K_{\epsilon})^-_{g_y}$ and therefore $k \in (K_{\epsilon})_{g_x} (K_{\epsilon})^-_{g_y}$. Otherwise, we can write $k_{22} = (\varpi + \epsilon y) z' + \epsilon z$ for some $z' \in \calO_{\epsilon}$ and $z \in \calO$. If $\frakv (z) = 0$, then we have $u_{\epsilon}^{-1} w^{-1} t^{-1}_{z^{-1} k_{12}} k \in (K_{\epsilon})^-_{g_y}$ and therefore $k \in (K_{\epsilon})_{g_x} w u_{\epsilon} (K_{\epsilon})^-_{g_y}$ (notice that in our case $k_{12} \in \calO_{\epsilon}^{\times}$). Finally, otherwise we have in fact $k_{22} = (\varpi + \epsilon y) (z' + \epsilon \varpi^{-1} z) \in (\varpi + \epsilon y) \calO_{\epsilon}$ and therefore $w^{-1} k \in (K_{\epsilon})^-_{g_y}$ and so $k \in (K_{\epsilon})_{g_x} w (K_{\epsilon})^-_{g_y}$.
	
	\medskip
	
	From here we get, by Lemma \ref{lem general hecke}, $$ K_{\epsilon} g_x K_{\epsilon} g_y K_{\epsilon} = K_{\epsilon} g_x g_y K_{\epsilon} \cup K_{\epsilon} g_x w g_y K_{\epsilon} \cup K_{\epsilon} g_x w u_{\epsilon} g_y K_{\epsilon}.$$ Further, we notice that $g_x w g_y = g_{x,y} w $ and that  $g_x w u_{\epsilon} g_y = u_x h_{x,y} \theta(u_y)$, showing that $$ K_{\epsilon} g_x w g_y K_{\epsilon} = K_{\epsilon} g_{x,y} K_{\epsilon}, \quad K_{\epsilon} g_x w u_{\epsilon} g_y K_{\epsilon} = K_{\epsilon} h_{x,y} K_{\epsilon}$$ and so indeed (\ref{eq decomposition double classes}) holds.
	
	\medskip
	
	We claim now that $$ K_{\epsilon} g_x g_y K_{\epsilon} \notin \left\{ K_{\epsilon} g_{x,y} K_{\epsilon}, \ K_{\epsilon} h_{x,y} K_{\epsilon} \right\}.$$ Indeed, reducing modulo $\epsilon$, it is enough to see that $$ K \mtrx{\varpi^2}{0}{0}{1} K \notin \{ K \mtrx{\varpi}{0}{0}{\varpi} K, \ K \mtrx{-\varpi (x+y)}{\varpi}{\varpi}{0} K \}.$$ To that end, notice that both double classes on the right are equal to $K$, while the double class on the left is not equal to $K$.

	\medskip
	
	Finally, we need to check whether $ K_{\epsilon} g_{x,y} K_{\epsilon}$ is equal to $K_{\epsilon} h_{x,y} K_{\epsilon}$ or not. Notice that $\det (h_{x,y}) = - \det (g_{x,y})$ and therefore for this computation we can work in ${\rm GL}_2$ instead of ${\rm PGL}_2$. Given $a \in K_{\epsilon}$, we compute $$ h_{x,y} a = \mtrx{(a_{21} - a_{11} x) (1 - \epsilon \varpi^{-1} x) - a_{11} y}{(a_{22} - a_{12} y) (1 - \epsilon \varpi^{-1} y) - a_{12} x}{ a_{11} + \epsilon \varpi^{-1} (a_{21} - a_{11} x)}{a_{12} + \epsilon \varpi^{-1} (a_{22} - a_{12} y)} g_{x,y}.$$ We need to check whether the first matrix on the right can belong to $K_{\epsilon}$ for some choice of $a$. It is easy to see that the belonging of this matrix to $K_{\epsilon}$ is equivalent to the conditions $$ (a_{21})_{[0]} - (a_{11})_{[0]} x \in \varpi \calO, \quad (a_{22})_{[0]} - (a_{12})_{[0]} y \in \varpi \calO.$$ If $y - x \in \varpi \calO$, these conditions imply that $\det (a)_{[0]} \in \varpi \calO$ and this contradicts $a$ belonging to $K_{\epsilon}$. Therefore if $y - x \in \varpi \calO$ our double classes are indeed distinct. If $y - x \notin \varpi \calO$, these conditions can be clearly met, and so in this case our double classes are not distinct.
\end{proof}

\subsubsection{}

We will not use the following claim in what follows, but we state it for completeness.

\begin{claim}\label{clm hecke computation}
	Let $x,y \in \kappa$. We have $$ T_x \star T_y = \begin{cases} q^{-2} \cdot {\rm ch}_{K_{\epsilon} g_x g_y K_{\epsilon}} + (1 + q^{-1}) \cdot {\rm ch}_{K_{\epsilon} g_{x,y} K_{\epsilon}} + q^{-1} \cdot {\rm ch}_{K_{\epsilon} h_{x,y} K_{\epsilon}} & x = y \\ q^{-2} \cdot {\rm ch}_{K_{\epsilon} g_x g_y K_{\epsilon}} + q^{-1} \cdot {\rm ch}_{K_{\epsilon} g_{x,y} K_{\epsilon}} & x \neq y \end{cases}.$$
\end{claim}

\begin{proof}

Denoting $$ n_k := ({\rm ch}_{K_{\epsilon} g_x K_{\epsilon}} \star {\rm ch}_{K_{\epsilon} g_y K_{\epsilon}}) (g_x k g_y),$$ we need to compute $n_1$ and $n_w$, and also $n_{w u_{\epsilon}}$ if $x = y$. To perform the computation we use the formula of Lemma \ref{lem general hecke}, which says here that $n_k$ is equal to the amount of pairs $(a,b) \in K_{\epsilon}^2$, where $a$ is taken modulo right multiplication by $(K_{\epsilon})^-_{g_x}$ and $b$ is taken modulo left multiplication by $(K_{\epsilon})_{g_y}$, such that $(g_y b g_y^{-1}) k^{-1} (g_x^{-1} a g_x) \in K_{\epsilon}$. Notice that it is easy to see how we can work in ${\rm GL}_2$ rather than ${\rm PGL}_2$ for this computation, and we shall do that.

\medskip

\underline{$k = 1$}: We calculate $$ (g_y b g_y^{-1}) \cdot 1 \cdot (g_x^{-1} a g_x ) = \mtrx{a_{11} b_{11} + (\varpi+\epsilon x) (\varpi + \epsilon y) a_{21} b_{12} }{\tfrac{a_{12} b_{11}}{\varpi + \epsilon x} + (\varpi + \epsilon y) a_{22} b_{12}}{\tfrac{a_{11} b_{21}}{\varpi + \epsilon y} + (\varpi + \epsilon y) a_{21} b_{22}}{\tfrac{a_{12} b_{21}}{(\varpi + \epsilon x) (\varpi + \epsilon y)} + a_{22} b_{22}}.$$ This belongs to $K_{\epsilon}$ if and only if $$ a_{12} \in (\varpi + \epsilon x) \calO_{\epsilon} \quad\textnormal{and}\quad b_{21} \in (\varpi + \epsilon y) \calO_{\epsilon},$$ i.e. if and only if $$ a \in K_{g_x}^- \quad\textnormal{and}\quad b \in K_{g_y}.$$ This shows that $n_1 = 1$.

\medskip

\underline{$k = w$}: We calculate $$ (g_y b g_y^{-1}) \cdot w^{-1} \cdot (g_x^{-1} a g_x ) = \mtrx{(\varpi + \epsilon x) a_{21} b_{11} + (\varpi + \epsilon y) a_{11} b_{12}}{a_{22} b_{11} + \frac{\varpi + \epsilon y}{\varpi + \epsilon x} a_{12} b_{12}}{a_{11} b_{22} + \frac{\varpi + \epsilon x}{\varpi + \epsilon y} a_{21} b_{21} }{\frac{a_{22} b_{21}}{\varpi + \epsilon y} + \frac{a_{12} b_{22}}{\varpi + \epsilon x}}.$$ Let us analyze several cases:

\begin{enumerate}
	\item Suppose that $a = u_{z_1}$ and $b = \theta (u_{z_2})$, where $z_1 , z_2 \in \calO_{\epsilon}$. Then $$ (g_y b g_y^{-1}) \cdot w^{-1} \cdot (g_x^{-1} a g_x ) = w u_{(\varpi + \epsilon y)^{-1} z_2 + (\varpi + \epsilon x)^{-1} z_1}.$$ Thus we need to count the amount of pairs $(z_1 , z_2)$, where $z_1$ is taken modulo $(\varpi + \epsilon x) \calO_{\epsilon}$ and $z_2$ is taken modulo $(\varpi + \epsilon y) \calO_{\epsilon}$, for which $(\varpi + \epsilon y)^{-1} z_2 + (\varpi + \epsilon x)^{-1} z_1 \in \calO_{\epsilon}$. If $y - x \in \varpi \calO$, we count that we have $q^2$ such pairs. If $y - x \notin \calO$, we count that we have $q$ such pairs.
	\item Suppose that $a = w u_{\epsilon z_1}$ and $b = \theta (u_{\epsilon z_2}) w$, where $z_1 , z_2 \in \calO$. Then $$ (g_y b g_y^{-1}) \cdot w^{-1} \cdot (g_x^{-1} a g_x ) = \mtrx{0}{\tfrac{\varpi + \epsilon y}{\varpi + \epsilon x}}{\tfrac{\varpi + \epsilon x}{\varpi + \epsilon y}}{\tfrac{\epsilon z_1}{\varpi + \epsilon y} + \tfrac{\epsilon z_2}{\varpi + \epsilon x}}.$$ We need to count the amount of pairs $(z_1 , z_2) \in \calO^2$, where $z_1$ and $z_2$ are taken modulo $\varpi \calO$, for which the matrix on the right has entries in $\calO_{\epsilon}$. If $y - x \notin \calO$, we count $0$ such pairs, while if $y - x \in \calO$ we count $q$ such pairs.
	\item Suppose that $a = u_{z_1}$ and $b = \theta (u_{\epsilon z_2}) w$, where $z_1 \in \calO_{\epsilon}$ and $z_2 \in \calO$. Then $$ (g_y b g_y^{-1}) \cdot w^{-1} \cdot (g_x^{-1} a g_x ) = \mtrx{*}{*}{*}{\tfrac{1}{\varpi + \epsilon y} + \tfrac{\epsilon z_1 z_2}{\varpi + \epsilon x}}.$$ The matrix on the right is never with entries in $\calO_{\epsilon}$. 
	\item In the case when $a = w u_{\epsilon z_1}$ and $b = \theta (u_{z_2})$, where $z_2 \in \calO$ and $z_2 \in \calO_{\epsilon}$, we analogously to the previous case obtain $0$ pairs.
\end{enumerate}
	Summing, we see that $$ n_w = \begin{cases} q^2 + q & x = y \\ q & x \neq y \end{cases}.$$

\medskip

\underline{$k = w u_{\epsilon}$}: Let us again analyze several cases:

\begin{enumerate}
	\item Suppose that $a = u_{z_1}$ and $b = \theta (u_{z_2})$, where $z_1 , z_2 \in \calO_{\epsilon}$. Then $$ (g_y b g_y^{-1}) \cdot (w u_{\epsilon})^{-1} \cdot (g_x^{-1} a g_x ) = \mtrx{-\epsilon}{1 - \tfrac{\epsilon z_1}{\varpi + \epsilon x}}{1 - \tfrac{\epsilon z_2}{\varpi + \epsilon y}}{\tfrac{z_1}{\varpi + \epsilon x} + \tfrac{z_2}{\varpi + \epsilon y}- \tfrac{\epsilon z_1 z_2}{(\varpi + \epsilon x)(\varpi + \epsilon y)}}.$$ We see from here that in order to have pairs as desired, we must have $(z_1)_{[0]} , (z_2)_{[0]} \in \varpi \calO$, and in such a case we have $q$ pairs.
	\item Suppose that $a = w u_{\epsilon z_1}$ and $b = \theta (u_{\epsilon z_2}) w$, where $z_1 , z_2 \in \calO$. Then $$ (g_y b g_y^{-1}) \cdot (w u_{\epsilon})^{-1} \cdot (g_x^{-1} a g_x ) = \mtrx{*}{*}{*}{\tfrac{\epsilon z_1}{\varpi + \epsilon y} + \tfrac{\epsilon z_2}{\varpi + \epsilon x} - \tfrac{\epsilon}{(\varpi + \epsilon x)(\varpi + \epsilon y)}}.$$ The matrix on the right is never with entries in $\calO_{\epsilon}$.
	\item Suppose that $a = u_{z_1}$ and $b = \theta (u_{\epsilon z_2}) w$, where $z_1 \in \calO_{\epsilon}$ and $z_2 \in \calO$. Then $$ (g_y b g_y^{-1}) \cdot (w u_{\epsilon})^{-1} \cdot (g_x^{-1} a g_x ) = \mtrx{*}{*}{1 - \tfrac{\epsilon}{\varpi + \epsilon y}}{*}.$$ The matrix on the right is never with entries in $\calO_{\epsilon}$.
	\item Suppose that $a = w u_{\epsilon z_1}$ and $b = \theta (u_{z_2})$, where $z_2 \in \calO$ and $z_2 \in \calO_{\epsilon}$. Then $$ (g_y b g_y^{-1}) \cdot (w u_{\epsilon})^{-1} \cdot (g_x^{-1} a g_x ) = \mtrx{z_1 - \tfrac{\epsilon}{\varpi + \epsilon x}}{\epsilon }{*}{*}.$$ The matrix on the right is never with entries in $\calO_{\epsilon}$.
\end{enumerate}
	Summing, we see that $n_{w u_{\epsilon}} = q$.
\end{proof}

\section{The spaces of $K_{\epsilon}$-invariants}

\subsection{The irreducible representations of $G_{\epsilon}$}

\subsubsection{}

We identify $\frakg$ with its Pontryagin dual $\frakg^{\vee}$ using the pairing $$(m_1 , m_2) \mapsto \psi_0 ({\rm tr} (m_1 m_2)).$$ Given $m \in \frakg$ we denote by $\psi_m : \frakg \to {\rm U} (1)$ the corresponding unitary character resulting from this identification.

\subsubsection{}

Let $m \in \frakg$ and let $W$ be an irreducible smooth representation of $Z_G (m)$. We consider $W$ as a representation of $(1 + \epsilon \frakg ) Z_G (m)$, using $$ (1+\epsilon m') z w := \psi_m (m') zw \quad (m' \in \frakg, \ z \in Z_G (m), \ w \in W).$$ We then define $$ V_{m , W} := {\rm ind}^{(1+\epsilon \frakg) Z_G (m)}_{G_{\epsilon}} W.$$ A smooth version of Mackey theory\footnote{A technical condition for the theory to hold is that the topological quotient $G \backslash \frakg^{\vee}$ is a $T_0$ topological space, which does indeed hold.} yields that, as we run over pairs $(m,W)$, the smooth $G_{\epsilon}$-representations $V_{m , W}$ are all irreducible, every irreducible smooth $G_{\epsilon}$-representation is isomorphic to one of them, and $V_{m,W}$ is isomorphic to $V_{m',W'}$ if and only if there exists $g \in G$ such that $m'$ is equal to $g m g^{-1}$ and $W'$ is isomorphic to $g W g^{-1}$. Let us also note here that every irreducible unitary representation of $G_{\epsilon}$ is isomorphic to the completion of some $V_{m,W}$, where $W$ is an irreducible smooth representation of $Z_G (m)$ carrying a unitary structure - this follows from taking into consideration the unitary version of Mackey theory\footnote{The condition which in \cite[\S 14]{Ma} is called regularity of the semidirect product, holds in our case.} (\cite[\S 14]{Ma}), as well as the knowledge that every irreducible unitary representation of $Z_G (m)$ is the completion of an irreducible smooth representation of $Z_G (m)$ carrying a unitary structure (all stabilizers $Z_G (m)$ are abelian, except $Z_G (0) = G$ which is reductive, in which case this follows from \cite{Be}).

\subsubsection{}

Concretely, we will think of $V_{m, W}$ as the space of smooth functions $f : G_{\epsilon} \to W$ satisfying the following properties:

\begin{enumerate}
	\item $f((1+\epsilon m') g) = \psi_m (m') f(g)$ for all $g \in G$ and $m' \in \frakg$.
	\item $f(zg) = z f(g)$ for all $g \in G$ and $z \in Z_G (m)$.
	\item The support of $f|_G$ is compact modulo the action of $Z_G (m)$ by left translation.
\end{enumerate}

The action of $G_{\epsilon}$ on $V_{m, \chi}$ is then given by $(gf)(g') := f(g'g)$ for $g,g' \in G_{\epsilon}$.

\subsubsection{}

Given a character $\chi$ of $Z_G (m)$, we will denote $V_{m , \chi} := V_{m , \bbC_{\chi}}$, where $\bbC_{\chi}$ is the one-dimensional $Z_G (m)$-representation whose underlying space is $\bbC$ and on which $Z_G (m)$ acts by $\chi$.

\subsection{The space of $K_{\epsilon}$-invariants}

\subsubsection{} We define:

\begin{definition} Let $m \in \frakg$.
	\begin{enumerate}
		\item We denote $$ G_m := \{ g \in G \ | \ g^{-1} m g \in \frakg_{\calO} \} \subset G.$$ Notice that $G_m$ is a clopen subset of $G$ which is invariant under $K$ on the right and under $Z_G (m)$ on the left.
		\item We denote $$ A_m := \{ gK \in G / K \ | \ g \in G_m \}$$ (it is an \textbf{affine Springer fiber}). Notice that $A_m$ is invariant under $Z_G (m)$ on the left.
	\end{enumerate} 
\end{definition}

\subsubsection{}\label{sssec invariants ident}

Let $m \in \frakg$ and let $W$ be an irreducible smooth representation of $Z_G (m)$. By identifying functions on $G_{\epsilon}$ which are in $V_{m , W}$ with their restrictions to $G$, we obtain an identification of $V_{m , W}^{K_{\epsilon}}$ with the space of smooth functions $f : G \to W$, which satisfy the following four properties:

\begin{enumerate}
	\item $f (zg) = z f(g)$ for all $g \in G$ and $z \in Z_G (m)$.
	\item The support of $f$ is compact modulo the action of $Z_G (m)$ by left translation.
	\item $f$ is $K$-invariant on the right.
	\item $f(g) = 0$ whenever $g \notin G_m$.
\end{enumerate}

\subsubsection{}\label{sssec invariants ident 2}

By thinking of functions on $G$ which are $K$-invariant on the right as functions on $G/K$, and restricting to $A_m$, we can also further identify $V_{m , W}^{K_{\epsilon}}$ with the space of functions $f : A_m \to W$ which satisfy the following two properties:

\begin{enumerate}
	\item $f (z a) = z f (a)$ for all $a \in A_m$ and $z \in Z_G (m)$.
	\item The support of $f$ is finite modulo the action of $Z_G (m)$ by left translation.
\end{enumerate}

\subsubsection{}

We have:

\begin{lemma}
	Let $m \in \frakg$ and let $W$ be an irreducible smooth $Z_G (m)$-representation. If $\frakv (\det (m)) < 0$ then $V_{m, W}^{K_{\epsilon}} = 0$.
\end{lemma}

\begin{proof}
	This is clear from condition (4) in \S\ref{sssec invariants ident}, since $G_m = \emptyset$ if $\frakv (\det (m)) < 0$.
\end{proof}

\subsection{The action of $T_x$ on the space of $K_{\epsilon}$-invariants}

\subsubsection{} We have:

\begin{lemma}
	Let $m \in \frakg$ and let $W$ be an irreducible smooth $Z_G (m)$-representation. If $\frakv (\det (m)) \ge 1$ then the operator by which $T_x$ acts on $V_{m, W}^{K_{\epsilon}}$ does not depend on $x \in \kappa$.
\end{lemma}

\begin{proof}
	Let us denote by $$ Q : V_{m,W} \to V_{m,W}$$ the projection operator given by pointwise multiplication by the characteristic function of $G_m$ and let us denote by $$ P : V_{m,W} \to V_{m,W}^{K_{\epsilon}}$$ the projection operator given by $K_{\epsilon}$-averaging.
	
	\medskip
	
	Notice that $P = P \circ Q$. Indeed, it is enough to see that if the support of $f \in V_{m,W}$ lies in $G \smallsetminus G_m$ then $P (f) = 0$. Since $G \smallsetminus G_m$ is $K$-invariant on the right, we have $P (f)|_{G_m} = 0$, and by condition (4) of \S\ref{sssec invariants ident}, this implies that $P (f) = 0$.
	
	\medskip
	
	Notice that ${\rm vol}_{G_{\epsilon}} (K_{\epsilon} g_x K_{\epsilon}) = {\rm vol}_{G_{\epsilon}} (K_{\epsilon} g_0 K_{\epsilon})$. Therefore, in order to see that $T_x = T_0$ it is enough to see that $P (g_x f) = P (g_0 f)$ for all $f \in V_{m,W}^{K_{\epsilon}}$, and for that, in view of the equality $P = P \circ Q$, it is enough to see that $Q ( g_x f) = Q (g_0 f)$ for all $f \in V_{m,W}^{K_{\epsilon}}$. In other words, we want to see that, given $f \in V_{m,W}^{K_{\epsilon}}$, we have $(g_x f) (g) = (g_0 f) (g)$ for all $g \in G_m$.
	
	\medskip
	
	We first calculate: \begin{equation}\label{eq gx in terms of g0} (g_x f) (g) = \psi_0 (\varpi^{-1} x \cdot (g^{-1} m g)_{11})\cdot (g_0 f) (g).\end{equation} Now, if $\frakv ((g^{-1} mg)_{11}) \ge 1$, then (\ref{eq gx in terms of g0}) shows that $(g_x f) (g) = (g_0 f) (g)$. In the other case, when $\frakv ((g^{-1} mg)_{11}) = 0$, we claim that $g g_0 \notin G_m$, which will imply that $(g_0 f) (g) = 0$ and from (\ref{eq gx in terms of g0}) it is then clear that $(g_x f) (g) = (g_0 f) (g) = 0$, finishing the proof. To see that $g g_0 \notin G_m$, notice that, since $\frakv ((g^{-1} mg)_{11}) = 0$ and $\frakv (\det (m)) \ge 1$, we have $\frakv ((g^{-1} mg)_{12}) = 0$. Therefore $$ \frakv (((g g_0)^{-1} m (g g_0))_{12}) = \frakv (\varpi^{-1} (g^{-1} m g)_{12}) = -1$$ and hence $(g g_0)^{-1} m (g g_0) \notin \frakg_{\calO}$, i.e. $g g_0 \notin G_m$.
\end{proof}

\subsubsection{}

Let $m \in \frakg$ and let $W$ be an irreducible smooth representation of $Z_G (m)$. Thinking about $V_{m,W}^{K_{\epsilon}}$ in terms of the description of \S\ref{sssec invariants ident 2}, let us describe the action of $T_0$ on it. Given $g \in G$, let us define
$$ (A_m \times A_m)_{g} := \{ (g_1 K , g_2 K) \in A_m \times A_m \ | \ g_2^{-1} g_1 \in K gK \}$$ and let us consider the two projections
$$
\xymatrix{ & (A_m \times A_m)_{g} \ar[rd]^{p^g_2} \ar[ld]_{p^g_1} & \\ A_m & & A_m }.
$$

\begin{claim}
When identifying $V_{m,W}^{K_{\epsilon}}$ with a space of functions on $A_m$ as in \S\ref{sssec invariants ident 2}, the action of $T_0$ on it is equal to $(p^{g_0}_2)_* \circ (p^{g_0}_1)^*$.
\end{claim}

\begin{proof}
	Let us fix $f \in V_{m,W}^{K_{\epsilon}}$, of which we think as a function $G/K$, or as a function on $G$ if we need to work inside $V_{m,W}$ and not just $V_{m,W}^{K_{\epsilon}}$. First, let us notice:
	\begin{equation}\label{eq hecke corres aff sp}(p^{g_0}_2)_* ((p^{g_0}_1)^* (f)) (hK) = \sum_{\substack{gK \in A_m \\ h^{-1} g \in K g_0 K }} f(gK).\end{equation}
	Let us now write $$ K_{\epsilon} g_0 K_{\epsilon} = \coprod_i k_i g_0 K_{\epsilon},$$ where $k_i$ are representatives in $K_{\epsilon}$ for $K_{\epsilon} / (K_{\epsilon})^-_{g_0}$. Let us write $k_i = (1 + \epsilon m_i) h_i$ where $m_i \in \frakg_{\calO}$ and $h_i \in K$. Notice that, for $f \in V_{m,W}^{K_{\epsilon}}$ and $h \in G_m$, $$ (k_i g_0 f) (h) = f (h k_i g_0) = f(h (1 + \epsilon m_i) h_i g_0) = $$ $$ = f((1 + \epsilon (h m_i h^{-1})) h h_i g_0 )  = \psi_{h^{-1}m h} (m_i) f(h h_i g_0) = f(h h_i g_0),$$ where the last step holds since $h^{-1} m h \in \frakg_{\calO}$. Therefore, we obtain the following formula: $$ (T_0 f) (hK) = q^{-1} \sum_i f(h h_i g_0).$$ Notice now that, as we run over the $h_i$'s, we run in a $q$-fold way over representatives in $K$ for $K / K^-_{g_0}$ (i.e. each representative appears precisely $q$ times). Namely, this is readily seen if we take the $k_i$'s to consist of $u_y$ where $y$ runs over representatives in $\calO_{\epsilon}$ for $\calO_{\epsilon} / \varpi \calO_{\epsilon}$ and $w u_{\epsilon z}$ where $z$ runs over representatives in $\calO$ for $\calO / \varpi \calO$.
	
	\medskip
	
	Thus, as we run over the $h_i g_0$'s, we run in a $q$-fold way over representatives in $K g_0 K$ for $K g_0 K / K$. Thus, we can write $$ (T_0 f) (hK) = q^{-1} \cdot q \cdot \sum_{g \in K g_0 K / K} f(h g) = \sum_{\substack{gK \in G / K \\ h^{-1} g \in K g_0 K}} f(gK) = \sum_{\substack{gK \in A_m \\ h^{-1} g \in K g_0 K}} f(gK),$$ and in view of (\ref{eq hecke corres aff sp}) we are done.
\end{proof}

\section{Explicit calculations}\label{sec explicit}

\subsection{Case of $m = 0$}

\subsubsection{}

The representations $V_{0,W}$ simply correspond to irreducible smooth representations $W$ of $G$, with $V_{0,W} = W$ on which $G_{\epsilon}$ acts via the projection $G_{\epsilon} \to G$, and so $V_{0,W}^{K_{\epsilon}} \cong W^K$ and the $T_x$'s act as the usual Hecke operator $q^{-1} \cdot {\rm ch}_{K t_{\varpi} K}$ acts on $W^K$, and therefore we omit this well-known case. In particular, the spaces of $K_{\epsilon}$-invariants in this case are all either $0$-dimensional or $1$-dimensional.

\subsection{Case of split non-zero $m$}\label{ssec calc split}

\subsubsection{}

Let us fix $c \in F \smallsetminus \{ 0 \}$ satisfying $\frakv (c) \ge 0$, and consider $$ m_c := \mtrx{c}{0}{0}{-c}.$$ Denoting by $T \subset G$ the subgroup of diagonal matrices, we have $Z_G (m_c) = T$.

\subsubsection{} We have:

\begin{lemma}
	 The affine Springer fiber $A_{m_c} \subset G / K$ consists of elements $$ u_b t_{\varpi^r} K$$ where $r \in \bbZ$ and $b \in F$ satisfies $\frakv (b) \ge r - \frakv (c)$. Two such elements $$ u_b t_{\varpi^r} K, \ u_{b'} t_{\varpi^{r'}} K$$  are equal if and only if $r' = r$ and $\frakv (b' - b) \ge r$.
\end{lemma}

\begin{proof}
	An easy calculation, using the decomposition $G = BK$.
\end{proof}

\begin{claim}\label{clm Am split}
	Every orbit of $T$ on $A_{m_c}$ contains an element from the family $\{ u_{\varpi^{-r}} K \}_{r \leq \frakv (c)}$. Two such elements $u_{\varpi^{-r}} K$ and $u_{\varpi^{-r'}} K$ lie in the same $T$-orbit if either $r = r'$ or both $r \leq 0$ and $r' \leq 0$. And, more precisely, if both $r \leq 0$ and $r' \leq 0$ then we have $u_{\varpi^{-r}} K = u_{\varpi^{-r'}} K$.
\end{claim}

\begin{proof}
	An easy calculation.
\end{proof}

\begin{definition}
	We define $$ T_{(0)} := \{ t_x : \ x \in F^{\times} , \ \frakv(x) = 0 \} \subset T$$ and for $\ell \in \bbZ_{\ge 1}$ we define $$ T_{(\ell)} := \{ t_x : \ x \in F^{\times} , \ \frakv (x - 1) \ge \ell \}  \subset T.$$
\end{definition}

\begin{lemma}
	Let $r \ge 0$. The stabilizer in $T$ of $u_{\varpi^{-r}} K$ is $T_{(r)}$.
\end{lemma}

\begin{proof}
	An easy calculation.
\end{proof}

\subsubsection{}

As a corollary of our calculations, we get:

\begin{corollary}\label{cor K inv split}
	Let $\chi$ be a character of $T$ and let us denote by $\frakc (\chi)$ its conductor, i.e. the smallest $\ell \in \bbZ_{\ge 0}$ for which $\chi$ is trivial on $T_{(\ell)}$. Given $f \in V_{m_c , \chi}$, let us denote by $\phi_f$ the function on $\bbZ$ given by $$ \phi_f (r) := f(u_{\varpi^{-r}}).$$ Then the association $f \mapsto \phi_f$ provides an isomorphism of vector spaces, between $V_{m_c , \chi}^{K_{\epsilon}}$ and the space of functions $\phi$ on $\bbZ$ satisfying:
	\begin{enumerate}
		\item $\phi (r_1) = \phi (r_2)$ for all $r_1, r_2 \leq 0$.
		\item $\phi(r) = 0$ for all $r > \frakv (c)$.
		\item $\phi(r) = 0$ for all $r < \frakc (\chi)$.
	\end{enumerate}

	In particular, $$ \dim V_{m_c, \chi}^{K_{\epsilon}} = \begin{cases} 0 & \frakc (\chi) > \frakv (c) \\ \frakv (c) + 1 - \frakc (\chi) & \frakc (\chi) \leq \frakv (c) \end{cases}.$$
\end{corollary}

\subsubsection{} Let us note:

\begin{remark}\label{rem Hecke op}
	Notice that, given $x \in \calO$, $m \in \frakg$ and and a smooth $G$-representation $W$, the action of $T_x$ on $V_{m, W}^{K_{\epsilon}}$ is equal to the action of $$ q^{-1} \sum_{y \in \calO_{\epsilon} / \varpi \calO_{\epsilon}} u_y g_x + q^{-1} \sum_{z \in \calO / \varpi \calO} w u_{\epsilon z} g_x.$$
\end{remark}

A calculation shows:

\begin{lemma}\label{lem formulas 1}
	Let $\chi$ be a character of $T$. Let $x \in \calO$, $y \in \calO_{\epsilon}$, $z \in \calO$, and let $r$ be an integer satisfying $\frakc (\chi) \leq r \leq \frakv (c)$.  Then, for $f \in V_{m_c , \chi}^{K_{\epsilon}}$:
	\begin{enumerate}
		\item Assume that either $r > 0$ or $r = 0$ and $y_{[0]} \notin -1 + \varpi \calO$. Then $$ (u_y g_x f) (u_{\varpi^{-r}}) = \psi_0 (2c \varpi^{-1} x) \chi (\varpi) f (u_{\varpi^{-(r+1)}}).$$
		\item Assume that $y_{[0]}\in -1 + \varpi \calO$. Then $$ (u_y g_x f) (u_{\varpi^{-0}}) = \psi_0 (2c \varpi^{-1} x) \chi (\varpi) f (u_{\varpi^{-0}}).$$
		\item We have $$ (w u_{\epsilon z} g_x f) (u_{\varpi^{-r}}) = \chi^{-1} (\varpi) f(u_{\varpi^{-(r-1)}}).$$
	\end{enumerate}
\end{lemma}

\subsubsection{}\label{sssec special rep}

Let us assume first that $\frakv (c) = 0$. Let $\chi$ be a character of $T$, such that $\frakc (\chi) \leq \frakv (c)$ i.e. $\frakc (\chi) = 0$. Then $V_{m_c, \chi}^{K_{\epsilon}}$ is one-dimensional, and using Lemma \ref{lem formulas 1} and Remark \ref{rem Hecke op} we find that $T_x$ acts on $V_{m_c, \chi}^{K_{\epsilon}}$ by the scalar
$$ \psi_0 (2c \varpi^{-1} x) \chi (\varpi) + \chi^{-1} (\varpi).$$

\subsubsection{}

Let us now consider the case when $\frakv (c) > 0$. As a corollary of Lemma \ref{lem formulas 1}, in view of Remark \ref{rem Hecke op}, we obtain:

\begin{claim}
	Assume $\frakv (c) > 0$. Let $\chi$ be a character of $T$. Let $x \in \kappa$ and let $r$ be an integer satisfying $ \frakc (\chi) \leq r \leq \frakv (c)$. Then, for $f \in V_{m_c , \chi}^{K_{\epsilon}}$, if $r > 0$: $$ (T_x f) (u_{\varpi^{-r}}) = q \chi (\varpi) \cdot f(u_{\varpi^{-(r+1)}}) + \chi^{-1} (\varpi) \cdot f(u_{\varpi^{-(r-1))}})$$ and, if $\frakc (\chi) = 0$:
	 $$ (T_x f) (u_{\varpi^{-0}}) = (q-1) \chi (\varpi) \cdot f(u_{\varpi^{-1}}) + (\chi (\varpi) + \chi^{-1} (\varpi)) \cdot f(u_{\varpi^{-0}}).$$
\end{claim}

Now, when $\frakv (c) > 0$ and given a character $\chi$ of $T$ for which $\frakc (\chi) \leq \frakv (c)$ let us consider the basis of $V_{m_c, \chi}^{K_{\epsilon}}$ consisting of elements $f_r$, for $\frakc (\chi) \leq r \leq \frakv (c)$, where $f_r$ is the unique function in $V_{m_c,\chi}^{K_{\epsilon}}$ satisfying $f_r (u_{\varpi^{-r}}) = q^{-(r-1)/2} (q-1)^{-1/2} \chi (\varpi)^{-r}$ if $r \in \bbZ_{\ge 1}$ and $f_r (u_1) = 1$ if $r = 0$, and satisfying $f_r (u_{\varpi^{-s}}) = 0$ for $s \in \bbZ_{\ge 0} \smallsetminus \{ r \}$. In the case when $\chi$ is unitary, the smooth $G_{\epsilon}$-representation $V_{m_c , \chi}$ has naturally a unitary structure (up to a positive scalar) and one readily checks that, for a suitable normalization of this unitary structure, the basis of $V_{m_c,\chi}^{K_{\epsilon}}$ consisting of $f_r$'s that we consider is orthonormal. The matrix representing $T_x$ with respect to that basis is as follows. If $\frakc (\chi) > 0$, the matrix is $$ \begin{pmatrix} 0 & q^{1/2} & 0 & 0 & \cdots & 0 \\ q^{1/2} & 0 & q^{1/2} & 0 & \cdots & 0 \\ 0 & q^{1/2} & 0 & \ddots & 0 & \vdots \\ 0 & 0 & \ddots & \ddots & q^{1/2} & 0 \\ \vdots & \vdots & 0 & q^{1/2} & 0 & q^{1/2} \\ 0 & 0 & 0 & 0 & q^{1/2} & 0 \end{pmatrix},$$ while if $\frakc (\chi) = 0$, the matrix is $$ \begin{pmatrix} \chi (\varpi) + \chi(\varpi)^{-1} & (q-1)^{1/2} & 0 & 0 & \cdots & 0 \\ (q-1)^{1/2} & 0 & q^{1/2} & 0 & \cdots & 0 \\ 0 & q^{1/2} & 0 & \ddots & 0 & \vdots \\ 0 & 0 & \ddots & \ddots & q^{1/2} & 0 \\ \vdots & \vdots & 0 & q^{1/2} & 0 & q^{1/2} \\ 0 & 0 & 0 & 0 & q^{1/2} & 0 \end{pmatrix}.$$

\subsection{Case of non-split $m$}\label{ssec calc nonsplit}

\subsubsection{}

Let us fix $d \in F$ such that $\frakv (d) \ge 0$ and $d$ is a non-square in $F$, and consider $$ m_d := \mtrx{0}{d}{1}{0}$$ (thus, we changed the notation $m_{-}$ from \S\ref{ssec calc split}). Then $$ Z_G (m_d) = \left\{ \mtrx{t}{ds}{s}{t} : \ (t,s) \in F^2 \smallsetminus \{ (0,0) \} \right\}.$$

\subsubsection{} We have:

\begin{lemma}
	 The affine Springer fiber $A_{m_d} \subset G / K$ consists of elements $$ u_b t_{\varpi^r} K$$ where $0 \leq r \leq \frakv (d)$ and $b \in F$ satisfies $\frakv (b) \ge r/2$. Two such elements $$ u_b t_{\varpi^r} K, \ u_{b'} t_{\varpi^{r'}} K$$  are equal if and only if $r' = r$ and $\frakv (b' - b) \ge r$.
\end{lemma}

\begin{proof}
	An easy calculation, using the decomposition $G = BK$.
\end{proof}

\begin{claim}\label{clm Am nonsplit}
Every orbit of $Z_G (m_d)$ on $A_{m_d}$ contains a unique element from the family $\left\{ t_{\varpi^r} K \right\}_{0 \leq r \leq \lfloor \frakv (d) / 2 \rfloor}$. Also, for $0 \leq r \leq \frakv (d)$, the elements $t_{\varpi^r} K$ and $t_{\varpi^{\frakv (d)-r}} K$ of $A_{m_d}$ lie in the same orbit of $Z_G (m_d)$; More precisely, we have $t_{\varpi^{\frakv (d) - r}} K = w t_{d^{-1}} \cdot t_{\varpi^r} K$.
\end{claim}

\begin{proof}
	Let us first see that, given $u_b t_{\varpi^r} K \in A_{m_d}$, so $0 \leq r \leq \frakv (d)$ and $\frakv (b) \ge r/2$, we have $u_b t_{\varpi^r} K \in Z_G (m_d) \cdot t_{\varpi^{r'}} K$ for some $0 \leq r' \leq \frakv (d)$. If $\frakv (b) \ge r$ then $ u_b t_{\varpi^r} K = t_{\varpi^r} u_{\varpi^{-r} b} K \in t_{\varpi^r} K$ and so we can assume that $\frakv (b) \leq r$. Then, denoting $r_0 := \min \{ 2 \frakv (b) , \frakv (d)\}$, we have $$ \mtrx{\tfrac{b}{\varpi^{r_0 - r}}}{\tfrac{d}{\varpi^{r_0 - r}}}{\tfrac{1}{\varpi^{r_0 - r}}}{\tfrac{b}{\varpi^{r_0 - r}}} t_{\varpi^{r_0 - r}} = u_b t_{\varpi^r} \mtrx{0}{\tfrac{d-b^2}{\varpi^{r_0}}}{1}{\tfrac{b}{\varpi^{r_0 - r}}},$$ showing that $u_b t_{\varpi^r} \in Z_G (m_d) \cdot t_{\varpi^{r_0 - r}} K$.

	\medskip

	Next, let us see that, for $0 \leq r \leq \frakv (d)$, we have
$$ t_{\varpi^{\frakv (d) - r}} K = w t_{d^{-1}} \cdot t_{\varpi^r} K.$$ Indeed, this follows from the equality
$$ \mtrx{0}{d \varpi^{-r}}{\varpi^{-r}}{0} t_{\varpi^r} = t_{\varpi^{\frakv (d) - r}} \mtrx{0}{\tfrac{d}{\varpi^{\frakv (d)}}}{1}{0}.$$

	\medskip

	Thus it is left to see that for $0 \leq r, r' \leq \lfloor \frakv (d) / 2 \rfloor$ with $r' \neq r$, we have $$ t_{\varpi^{r'}} \notin Z_G ({m_d}) \cdot t_{\varpi^r} K.$$ Notice that we have the following invariant of $gK \in A_{m_d}$: The smallest $\ell \in \bbZ_{\ge 0}$ for which $\varpi^{-\ell-1} \cdot g^{-1} m_d g \notin \frakg_{\calO}$. Notice that this invariant is invariant under the left action of $Z_G ({m_d})$ on $A_{m_d}$. One readily calculates that, for $0 \leq r \leq \lfloor \frakv (d) / 2 \rfloor$, this invariant is equal to $r$ on $t_{\varpi^r} K$, which yields the desired.
\end{proof}

\begin{definition}
	Given $\ell \in \bbZ_{\ge 0 }$, if $\ell \neq 0$ or if $\frakv (d)$ is odd and $\ell = 0$, we define $$ Z_G (m_d)_{(\ell)} := \left\{ \mtrx{1}{ds}{s}{1}: \ s \in F, \ \frakv (s) \ge \ell - \lfloor \frakv (d) / 2 \rfloor \right\} \subset Z_G (m_d),$$ and, if $\frakv (d)$ is even, we define $$ Z_G (m_d)_{(0)} := Z_G (m_d).$$
\end{definition}

\begin{lemma}
	The above-defined subsets $Z_G (m_d)_{(\ell)}$ of $Z_G (m_d)$ are open subgroups, we have $Z_G (m_d)_{(\ell_1)} \subset Z_G (m_d)_{(\ell_2)}$ whenever $\ell_1 \ge \ell_2$ and the intersection of all $Z_G (m_d)_{(\ell)}$ is equal to $\{ 1\}$. Let $0 \leq r \leq \lfloor \frakv (d) / 2 \rfloor$. The stabilizer in $Z_G (m_d)$ of $t_{\varpi^r} K$ is $Z_G (m_d)_{(\lfloor \frakv (d) / 2 \rfloor - r)}$.
\end{lemma}

\begin{proof}
	An easy calculation.
\end{proof}

\subsubsection{}

As a corollary of our calculations, we get:

\begin{corollary}
	Let $\chi$ be a character of $Z_G (m_d)$ and let us denote by $\frakc (\chi)$ the smallest $\ell \in \bbZ_{\ge 0}$ for which $\chi$ is trivial on $Z_G (m_d)_{(\ell)}$. Given $f \in V_{m_d , \chi}$, let us denote by $\phi_f$ the function on $\bbZ$ given by $$ \phi_f (r) := f(t_{\varpi^r}).$$ Then the association $f \mapsto \phi_f$ provides an isomorphism of vector spaces, between $V_{m_d , \chi}^{K_{\epsilon}}$ and the space of functions $\phi$ on $\bbZ$ satisfying:
	\begin{enumerate}
		\item $\phi (r) = 0$ if $r < 0$ or $r > \frakv (d)$.
		\item $\phi (\frakv (d) - r) = \chi (w t_{d^{-1}}) \phi (r)$ for all $r \in \bbZ$.
		\item $\phi (r) = 0$ if $ \lfloor \frakv (d) / 2 \rfloor - \frakc (\chi) < r \leq \lfloor \frakv (d) / 2 \rfloor$.
	\end{enumerate}

	In particular, $V_{m_d , \chi}^{K_{\epsilon}} = 0$ if $\frakc (\chi) > \lfloor \frakv (d) / 2 \rfloor$ and otherwise $$ \dim V_{m_d , \chi}^{K_{\epsilon}} = \lfloor \frakv (d) / 2 \rfloor + 1 - \frakc (\chi).$$
\end{corollary}

\subsubsection{} A calculation shows:

\begin{lemma} Let $\chi$ be a character of $Z_G (m_d)$. Let $x \in \calO$, $y \in \calO_{\epsilon}$, $z \in \calO$ and let $r$ be an integer satisfying $ 0 \leq r \leq \lfloor \frakv (d) / 2 \rfloor - \frakc (\chi)$. Then, for $f \in V_{m_d , \chi}^{K_{\epsilon}}$:
	\begin{enumerate}
		\item If $\frakv (y_{[0]}) = 0$:
$$ (u_y g_x f) (t_{\varpi^r}) =  f (t_{\varpi^{r-1}}).$$
		\item If $\frakv (y_{[0]}) > 0$:
$$ (u_y g_x f) (t_{\varpi^r}) = f (t_{\varpi^{r+1}}).$$
		\item We have
$$ (w u_{\epsilon z} g_x f ) (t_{\varpi^r}) = f (t_{\varpi^{r-1}}).$$
	\end{enumerate}
\end{lemma}

As a corollary of this lemma, in view of Remark \ref{rem Hecke op}, we obtain:

\begin{claim}
	 Let $\chi$ be a character of $Z_G (m_d)$. Let $x \in \kappa$ and let $r$ be an integer satisfying $ 0 \leq r \leq \lfloor \frakv (d) / 2 \rfloor - \frakc (\chi)$. Then, for $f \in V_{m_d , \chi}^{K_{\epsilon}}$: $$ (T_x f) (t_{\varpi^r}) = q \cdot f(t_{\varpi^{r-1}}) + f(t_{\varpi^{r+1}}).$$
\end{claim}

Now, given a character $\chi$ of $Z_G (m_d)$ for which $\frakc (\chi) \leq \lfloor \frakv (d) / 2 \rfloor$, let us consider the basis of $V_{m_d, \chi}^{K_{\epsilon}}$ consisting of elements $f_r$, for $0 \leq r \leq \lfloor \frakv (d) / 2 \rfloor - \frakc (\chi)$, where $f_r$ is the unique function in $V_{m_d,\chi}^{K_{\epsilon}}$ satisfying $f_r (t_{\varpi^{r}}) = q^{r/2}$ in the case when $r < \frakv (d) / 2$ and $f_{\frakv (d) / 2} (t_{\varpi^{\frakv (d) / 2}}) = (q+1)^{1/2} q^{(\frakv (d)/2 - 1)/2}$ in the case when $\frakv (d)$ is even and $\frakc (\chi) = 0$, and satisfying $f_r (t_{\varpi^{s}}) = 0$ for $s \in \bbZ_{\leq \lfloor \frakv (d) / 2 \rfloor} \smallsetminus \{ r \}$. The smooth $G_{\epsilon}$-representation $V_{m_d , \chi}$ has naturally a unitary structure (up to a positive scalar) and one readily checks that, for a suitable normalization of this unitary structure, the basis of $V_{m_d,\chi}^{K_{\epsilon}}$ consisting of $f_r$'s that we consider is orthonormal. The matrix representing $T_x$ with respect to that basis is as follows. If $\frakc (\chi) > 0$, the matrix is $$ \begin{pmatrix} 0 & q^{1/2} & 0 & 0 & \cdots & 0 \\ q^{1/2} & 0 & q^{1/2} & 0 & \cdots & 0 \\ 0 & q^{1/2} & 0 & \ddots & 0 & \vdots \\ 0 & 0 & \ddots & \ddots & q^{1/2} & 0 \\ \vdots & \vdots & 0 & q^{1/2} & 0 & q^{1/2} \\ 0 & 0 & 0 & 0 & q^{1/2} & 0 \end{pmatrix}.$$ If $\frakc (\chi) = 0$ and $\frakv (d)$ is odd, the matrix is $$ \begin{pmatrix} 0 & q^{1/2} & 0 & 0 & \cdots & 0 \\ q^{1/2} & 0 & q^{1/2} & 0 & \cdots & 0 \\ 0 & q^{1/2} & 0 & \ddots & 0 & \vdots \\ 0 & 0 & \ddots & \ddots & q^{1/2} & 0 \\ \vdots & \vdots & 0 & q^{1/2} & 0 & q^{1/2} \\ 0 & 0 & 0 & 0 & q^{1/2} & \chi (w t_{d^{-1}}) \end{pmatrix}.$$ If $\frakc (\chi) = 0$ and $\frakv (d)$ is even, the matrix is $$ \begin{pmatrix} 0 & q^{1/2} & 0 & 0 & \cdots & 0 \\ q^{1/2} & 0 & q^{1/2} & 0 & \cdots & 0 \\ 0 & q^{1/2} & 0 & \ddots & 0 & \vdots \\ 0 & 0 & \ddots & \ddots & q^{1/2} & 0 \\ \vdots & \vdots & 0 & q^{1/2} & 0 & (q+1)^{1/2} \\ 0 & 0 & 0 & 0 & (q+1)^{1/2} & 0 \end{pmatrix}.$$

\subsection{Case of nilpotent non-zero $m$}

\subsubsection{}

Let us denote $$ m := \mtrx{0}{1}{0}{0}.$$ Denoting by $U \subset G$ the subgroup of unipotent upper triangular matrices, we have $Z_G (m) = U$.

\subsubsection{} We have:

\begin{lemma}
	 The affine Springer fiber $A_{m} \subset G / K$ consists of elements $$ u_b t_{\varpi^r} K$$ where $r \leq 0$. Two such elements $$ u_b t_{\varpi^r} K, \ u_{b'} t_{\varpi^{r'}} K$$  are equal if and only if $r' = r$ and $\frakv (b' - b) \ge r$.
\end{lemma}

\begin{proof}
	An easy calculation, using the decomposition $G = BK$.
\end{proof}

\begin{claim}\label{clm Am nilp}
	Every orbit of $U$ on $A_{m}$ contains a unique element from the family $\{ t_{\varpi^{-r}} K \}_{r \ge 0}$.
\end{claim}

\begin{proof}
	An easy calculation.
\end{proof}

\begin{definition}
	Let $\ell \in \bbZ$. We define $$ U_{(\ell)} := \{ u_s : \ s \in F, \  \frakv (s) \ge \ell \} \subset U.$$
\end{definition}

\begin{lemma}
	Let $r \ge 0$. The stabilizer in $U$ of $t_{\varpi^{-r}} K$ is $U_{(-r)}$.
\end{lemma}

\begin{proof}
	An easy calculation.
\end{proof}

\subsubsection{} As a corollary of our calculations, we get:

\begin{corollary}
	Let $\psi$ be a character of $U$ and let us denote by $\frakc (\psi)$ the smallest $\ell \in \bbZ$ for which $\psi$ is trivial on $U_{(\ell)}$ (and set $\frakc (\psi) = -\infty$ if $\psi = 1$). Given $f \in V_{m , \psi}$, let us denote by $\phi_f$ the function on $\bbZ$ given by $$ \phi_f (r) := f(t_{\varpi^{-r}}).$$ Then the association $f \mapsto \phi_f$ provides an isomorphism of vector spaces, between $V_{m , \psi}^{K_{\epsilon}}$ and the space of functions $\phi$ on $\bbZ$ satisfying:
	\begin{enumerate}
		\item $\phi (r) = 0$ if $r < 0$.
		\item $\phi (r) = 0$ if $r > -\frakc (\psi)$.
	\end{enumerate}

	In particular, $V_{m , \psi}^{K_{\epsilon}} = 0$ if $\frakc (\psi) > 0$ and otherwise $$ \dim V_{m , \psi}^{K_{\epsilon}} = 1 - \frakc (\psi)$$ (and $\dim V_{m , \psi}^{K_{\epsilon}} = \infty$ if $\psi = 1$).
\end{corollary}

\subsubsection{} A calculation shows:

\begin{lemma} Let $\psi$ be a character of $U$. Let $x \in \calO$, $y \in \calO_{\epsilon}$, $z \in \calO$ and let $r$ be an integer satisfying $ 0 \leq r \leq - \frakc (\psi)$. Then, for $f \in V_{m , \psi}^{K_{\epsilon}}$:
	\begin{enumerate}
		\item We have $$ (u_y g_x f) (t_{\varpi^{-r}}) = f(t_{\varpi^{-(r-1)}}).$$
		\item We have $$ (w u_{\epsilon z} g_x f) (t_{\varpi^{-r}}) = f (t_{\varpi^{-(r+1)}}).$$
	\end{enumerate}
\end{lemma}

As a corollary of this lemma, in view of Remark \ref{rem Hecke op}, we obtain:

\begin{claim}
	 Let $\psi$ be a character of $U$. Let $x \in \kappa$ and let $r$ be an integer satisfying $ 0 \leq r \leq - \frakc (\psi)$. Then, for $f \in V_{m , \psi}^{K_{\epsilon}}$: $$ (T_x f) (t_{\varpi^{-r}}) = q \cdot f(t_{\varpi^{-(r-1)}}) + f(t_{\varpi^{-(r+1)}}).$$
\end{claim}

Now, given a character $\psi$ of $U$ for which $\frakc (\psi) \leq 0$ let us consider the basis of $V_{m,\psi}^{K_{\epsilon}}$ consisting of elements $f_r$, for $0 \leq r \leq -\frakc (\psi)$ (or $0 \leq r$ in the case $\psi = 1$), where $f_r$ is the unique function in $V_{m,\psi}^{K_{\epsilon}}$ satisfying $f_r (t_{\varpi^{-r}}) = q^{r/2}$ and $f_r (t_{\varpi^{-s}}) = 0$ for $s \in \bbZ \smallsetminus \{ r \}$. Being unitarily induced from a unitary character, the smooth $G_{\epsilon}$-representation $V_{m,\psi}$ has naturally a unitary structure (up to a positive scalar) and one readily checks that, for a suitable normalization of this unitary structure, the basis of $V_{m,\psi}^{K_{\epsilon}}$ consisting of $f_r$'s that we consider is orthonormal. The matrix representing $T_x$ with respect to that basis is $$ \begin{pmatrix} 0 & q^{1/2} & 0 & \cdots & 0 & \cdots \\ q^{1/2} & 0 & q^{1/2} & \cdots & 0 & \cdots \\ 0 & q^{1/2} & 0 & \ddots & 0 & \cdots \\ \vdots & \vdots & \ddots & \ddots & q^{1/2} & \cdots \\ 0 & 0 & \cdots & q^{1/2} & 0 & \ddots \\ \vdots & \vdots & \vdots & \vdots & \ddots & \ddots \end{pmatrix}.$$

\subsection{Description of the subalgebra of $\calH_{K_{\epsilon}} (G_{\epsilon})$ generated by the $T_x$'s}

\subsubsection{} We define:

\begin{definition}
	Let us denote by $\calT \subset \calH_{K_{\epsilon}} (G_{\epsilon})$ the subalgebra generated by $ \{T_x : \ x \in \kappa \}$.
\end{definition}

\subsubsection{} We have:

\begin{proposition}\label{prop shape of T}\
	\begin{enumerate}
		\item There exists a surjective algebra morphism $$ \beta_1 : \calT \to \bbC [z]$$ characterized by sending each $T_x$ to $z$. As we run over irreducible smooth representations $V$ of $G_{\epsilon}$, except those of \S\ref{sssec special rep}, and consider the corresponding action maps $\calT \to {\rm End} (V^{K_{\epsilon}})$, their joint kernel is equal to the kernel of $\beta_1$.
		\item Given\footnote{Here and later, $\kappa^{\vee}$ denotes the group of characters of $\kappa$.} $\psi \in \kappa^{\vee} \smallsetminus \{ 1 \}$, there exists a surjective algebra morphism $$ \beta_{\psi} : \calT \to \bbC [z , z^{-1}]$$ characterized by sending each $T_x$ to $\psi (x) z^{-1} + z$. Fixing $c \in \calO^{\times}$ for which $\psi_0 (2c \varpi^{-1} x) = \psi (x+\varpi \calO)$ for all $x \in \calO$, as we run over  irreducible smooth reprsentations $V$ of $G_{\epsilon}$ which appear in \S\ref{sssec special rep} with that specific $c$, and consider the corresponding action maps $\calT \to {\rm End} (V^{K_{\epsilon}})$, their joint kernel is equal to the kernel of $\beta_{\psi}$.
		\item The algebra morphism $$ \beta_1 \times \prod_{\psi \in \kappa^{\vee} \smallsetminus \{ 1\} } \beta_{\psi} : \calT \to \bbC [z] \times \prod_{\psi \in \kappa^{\vee} \smallsetminus \{ 1\} } \bbC [z , z^{-1}]$$ is an isomorphism.
	\end{enumerate}
\end{proposition}

\begin{proof}\

\begin{enumerate}

\item Let us notice that all irreducible smooth representations $V$ of $G_{\epsilon}$, except those of \S\ref{sssec special rep}, have the property that, under the corresponding action map $\beta_V : \calT \to {\rm End} (V^{K_{\epsilon}})$, the images of the $T_x$'s are all equal. Thus, unfolding things, one sees that in order to check everything, the only non-trivial part is to check that if a polynomial $p$ in one variable annihilates $\beta_V (T_0)$ for all $V$ as indicated, then $p = 0$. By noticing, for example, that matrices representing $\beta_V (T_0)$, written above, have jointly infinitely many eigenvalues (for example, one sees this using Proposition \ref{prop distribution}, but one can see this in simpler ways), the claim is clear.

\item Unfolding things, taking into consideration how the representations of \S\ref{sssec special rep} look like, we see that the statement, except the surjectivity of $\beta_{\psi}$, is simply that, given a polynomial $p$ in the variables $t_x$, and obtaining from it a Laurent polynomial $p'$ in the variable $z$ by substituting $\psi (x) z + z^{-1}$ in place of $t_x$, we have $p' = 0$ if and only if $p' (\zeta) = 0$ for all $\zeta \in \bbC^{\times}$, which is clear. To check the surjectivity of $\beta_{\psi}$, consider $x \in \kappa$ for which $\psi (x) \neq 1$. Then the image of $\beta_{\psi}$ contains both $z^{-1} + z$ and $\psi (x) z^{-1} + z$ and $z$ and $z^{-1}$ are linear combinations of these two elements, so the image of $\beta_{\psi}$ contain both $z$ and $z^{-1}$, and thus is the whole $\bbC [z , z^{-1}]$.

\item Given $\psi \in \kappa^{\vee}$, let us denote by $\calI_{\psi} \subset \calT$ the kernel of $\beta_{\psi}$. By the ``Chinese reminder theorem", it is enough to check that:
\begin{itemize}
	\item $\calI_{\psi_1} + \calI_{\psi_2} = (1)$ for all $\psi_1 , \psi_2 \in \kappa^{\vee}$ such that $\psi_1 \neq \psi_2$.
	\item The interesection $\cap_{\psi \in \kappa^{\vee}} \calI_{\psi}$ is equal to $\{ 0 \}$.
\end{itemize}

\medskip

To check the first condition, notice that for $\psi \in \kappa^{\vee}$ and $x \in \kappa$ we have $$ \left( \sum_{y \in \kappa} T_y \right) (T_0 - T_x) - q (1 - \psi (x)) \in \calI_{\psi}$$ and therefore for $\psi_1 , \psi_2 \in \kappa^{\vee}$ and $x \in \kappa$ we have $$ q (\psi_1 (x) - \psi_2 (x)) \in \calI_{\psi_1} + \calI_{\psi_2}.$$ Therefore, if $\psi_1 \neq \psi_2$ then (by considering $x \in \kappa$ for which $\psi_1 (x) \neq \psi_2 (x)$) we obtain $\calI_{\psi_1} + \calI_{\psi_2} = (1)$, as desired.

\medskip

To check the second condition, notice that, by (1) and (2), elements in $\cap_{\psi \in \kappa^{\vee}} \calI_{\psi}$ act on all irreducible smooth representations of $G_{\epsilon}$ by zero. As is well known (see for example \cite[III.1.11]{Re}), such elements must be zero, as desired.

\end{enumerate}

\end{proof}

\section{Spectral properties of our Hecke operators}

\subsection{Statements of the properties}\label{ssec properties}

\subsubsection{} We have:

\begin{proposition}\label{prop distribution}
	Let us consider pairs $(m , \chi)$ where $0 \neq m \in \frakg$, $\chi$ is a character of $Z_G (m)$ and such that:
	\begin{itemize}
		\item $V_{m , \chi}^{K_{\epsilon}}$ is finite-dimensional and the action of $T_x$ on it does not depend on $x$. In other words, if $m$ is nilpotent then we assume that $\chi \neq 1$, while if $m$ has an eigenvalue $c$ in $F \smallsetminus \{ 0 \}$ and $\frakc (\chi) = 0$ then we assume that $\frakv (c) \neq 0$.
		\item If $m$ has an eigenvalue in $F \smallsetminus \{ 0 \}$ and $\frakc (\chi) = 0$, then we assume that $| z | = 1$ for all $z \in {\rm Im} (\chi)$.
	\end{itemize}
	Then all eigenvalues of the action of $T_0$ on $V_{m , \chi}^{K_{\epsilon}}$ have algebraic multiplicity $1$, are real and lie in the interval $[-2 q^{1/2} , 2 q^{1/2}]$; denote by $\Lambda_{m , \chi}$ the set of these eigenvalues. Given a sequence $(m_n , \chi_n)$ for which $$\lim_{n \to \infty} \dim V_{m_n , \chi_n}^{K_{\epsilon}} = \infty$$ and given any $f \in C_c (\bbR)$ we have $$ \lim_{n \to \infty} \frac{1}{|\Lambda_{m_n , \chi_n}|} \sum_{\lambda \in \Lambda_{m_n , \chi_n}} f(\lambda / 2q^{1/2}) = \frac{1}{\pi} \int_{-1}^1 \frac{f(x)dx}{\sqrt{1-x^2}}.$$
\end{proposition}

\begin{proof}
	This follows from the explicit calculations of \S\ref{sec explicit} and Lemma \ref{lem matrix} formulated below.
\end{proof}

\subsubsection{}

We also have:

\begin{proposition}\label{prop weil}
	Let us consider a pair $(m , \chi)$ where $0 \neq m \in \frakg$, $\chi$ is a character of $Z_G (m)$ and such that:
	\begin{itemize}
		\item $V_{m , \chi}^{K_{\epsilon}}$ is finite-dimensional and the action of $T_x$ on it does not depend on $x$. In other words, if $m$ is nilpotent then we assume that $\chi \neq 1$, while if $m$ has an eigenvalue $c$ in $F \smallsetminus \{ 0 \}$ and $\frakc (\chi) = 0$ then we assume that $\frakv (c) \neq 0$.
		\item If $m$ has an eigenvalue in $F \smallsetminus \{ 0 \}$ and $\frakc (\chi) = 0$, then we assume that $z$ is a root of unity for all $z \in {\rm Im} (\chi)$.
	\end{itemize}
	Then every eigenvalue of the action of $T_0$ on $V_{m , \chi}^{K_{\epsilon}}$ is expressible as the sum of two complex-conjugate $q$-Weil numbers (algebraic integers all of whose conjugates have absolute value $q^{1/2}$).
\end{proposition}

\begin{proof}
	This follows from the explicit calculations of \S\ref{sec explicit} and Lemma \ref{lem matrix Weil} formulated below, for the application of which we use the the bound on eigenvalues obtained in Lemma \ref{lem matrix} below.
\end{proof}

\subsubsection{}

Regarding the case when we have an infinite-dimensional space of $K_{\epsilon}$-invariants:

\begin{proposition}\label{prop inf dim}
	Let $$ m := \mtrx{0}{1}{0}{0}.$$ Recall that $V_{m,1}$ carries a unitary structure, which we can normalize by requiring $f_0$ to have norm $1$, where $f_0 \in V_{m,1}^{K_{\epsilon}}$ is the unique element satisfying $f_0 (1) = 1$ and $f_0(t_{\varpi^{r}}) = 0$ for all $r \in \bbZ \smallsetminus \{ 0 \}$. There exists a unique isomorphism $\Phi$ of Hilbert spaces, between the completion of $V_{m,1}^{K_{\epsilon}}$ and $L^2 \left( [-2 q^{1/2}, 2 q^{1/2}] , (2 \pi q)^{-1} \sqrt{4q - x^2} dx \right)$, satisfying the following two properties:
	\begin{enumerate}
		\item $\Phi (T_0 f) = {\rm id} \cdot \Phi (f)$ for all $f \in V_{m,1}^{K_{\epsilon}}$, where ${\rm id}$ is the function on $[-2 q^{1/2}, 2 q^{1/2}]$ sending $x$ to $x$.
		\item $\Phi (f_0) = 1$.
	\end{enumerate}
\end{proposition}

\begin{proof}
	This follows from the explicit calculations of \S\ref{sec explicit} and Lemma \ref{lem inf dim spectrum} formulated below.
\end{proof}

\subsection{Lemmas for Proposition \ref{prop distribution}}

\subsubsection{} We have:

\begin{lemma}\label{lem matrix}
	Let $a, b , c_n \in \bbR$ and for every $n \in \bbZ_{\ge 1}$ consider the $n \times n$ matrix $$ A_n := \begin{pmatrix} c_n & b & 0 & 0 & \cdots & 0 \\ b & 0 & a & 0 & \cdots & 0 \\ 0 & a & 0 & \ddots & 0 & \vdots \\ 0 & 0 & \ddots & \ddots & a & 0 \\ \vdots & \vdots & 0 & a & 0 & a \\ 0 & 0 & 0 & 0 & a & 0 \end{pmatrix}.$$
	Assume that the following hold:
		\begin{enumerate}
			\item $|b| + |c_n| \leq 2|a|$.
			\item Either $|b| \leq |a|$ or $c_n = 0$ and $|b| \leq 2^{1/2} |a|$.
		\end{enumerate}
	Denote by $\{ \lambda^{(n)}_k \}_{k = 1}^n$ the eigenvalues of $A_n$. Then $\lambda^{(n)}_k \neq \lambda^{(n)}_{k'}$ whenever $k \neq k'$, all $\lambda^{(n)}_k$ are real and contained in the interval $[-2 |a| , 2 |a|]$, and for every $f \in C_c (\bbR)$ we have $$ \lim_{n \to \infty} \frac{1}{n} \sum_{k = 1}^n f(\lambda^{(n)}_k / 2|a|) = \frac{1}{\pi} \int_{-1}^1 \frac{f(x) dx}{\sqrt{1-x^2}}.$$
\end{lemma}

\begin{proof}
	Since $A$ is real and symmetric, all its eigenvalues are real, with algebraic multiplicity equal to geometric multiplicity. It is easy to check directly, by recursion (which we write more explicitly below), that an eigenvector of $A_n$ is linearly determined by its first coordinate, so that all geometric multiplicities are equal to $1$; hence $\lambda_k^{(n)} \neq \lambda_{k'}^{(n)}$ for $k \neq k'$. Let us see that for every eigenvalue $\lambda$ of $A_n$ we have $| \lambda | \leq 2|a|$. Let $(x_0 , \ldots , x_{n-1})$ be an eigenvector of $A_n$ with eigenvalue $\lambda$. We obtain: $$ (\lambda - c_n) x_0 = b x_1 , \quad \lambda x_1 = a x_2 + b x_0,$$ $$ \lambda x_2 = a x_3 + a x_1, \quad \ldots , \lambda x_{n-2} = a x_{n-3} + a x_{n-1} , \quad \lambda x_{n-1} = a x_{n-2}. $$ Let $0 \leq i \leq n-1$ be such that $|x_i| = \max_{0 \leq j \leq n-1} | x_j |$. We consider a few cases. If $2 \leq i \leq n-1$ then we obtain (denoting for convenience $x_n := 0$) $$ |\lambda| |x_i| \leq |a| (|x_{i-1}| + |x_{i+1}|) \leq 2|a| |x_i|$$ and thus $|\lambda | \leq 2|a|$, as desired. If $i = 0$, then we obtain $$ | \lambda - c_n | |x_0| = |b| | x_1 | \leq |b| |x_0|$$ and so $| \lambda | \leq |b| + |c_n| \leq 2|a|$. It is left to treat the case $i = 1$. If $|b| \leq |a|$ then it is treated similarly to the case $0 \leq i \leq n-1$. Otherwise, we have $c_n = 0$ and $|b| \leq 2^{1/2} |a|$. We calculate: $$ x_1 = \frac{\lambda}{b} x_0, \quad x_2 = \frac{\lambda^2 - b^2}{ab} x_0.$$ Thus, since $|x_2| \leq |x_1|$, we obtain $$ |\lambda^2 - b^2| \leq |a| |\lambda|.$$ This implies $$ |\lambda |^2 - |a| |\lambda|- b^2 \leq 0$$ which implies $$|\lambda| \leq \frac{|a| + \sqrt{a^2+4b^2}}{2} \leq \frac{|a| + \sqrt{a^2 + 4(2^{1/2} |a|)^2}}{2}  = 2|a|.$$ Let us now consider the distribution statement. If $c_n = 0$ and $b = a$, it is well-known (as we recall in Lemma \ref{lem basic matrix}) that the eigenvalues are given by $$ 2 |a| \cdot \cos \left( \tfrac{k \pi}{n+1}\right), \quad 1 \leq k \leq n,$$ and the statement is clear. The general case can be reduced to that special case by using \cite[Theorem 3.1]{Ty} (with $\Delta_n (\epsilon) := 0$ and $r(\epsilon) := 2$).
\end{proof}

\subsection{Lemmas for Proposition \ref{prop weil}}

\subsubsection{}

Given $q\in \bbZ_{>0}$, we will say that an algebraic integer $x$ is a \textbf{$q$-Weil number} if for all automorphisms $\sigma \in {\rm Aut} (\overline{\bbQ} / \bbQ)$ we have $|\sigma (x)| = q^{1/2}$. We will say that a real algebraic integer $x$ is a \textbf{$\Re$-$q$-Weil number} if it can be written as a sum of two complex-conjugate $q$-Weil numbers. Equivalently, if $|x| \leq 2 q^{1/2}$ and, for the unique algebraic number $\zeta$ of absolute value $1$ for which $x = q^{1/2} \zeta + q^{1/2} \zeta^{-1}$, all the conjugates of $\zeta$ have absolute value $1$. Finally, we will say that an algebraic number $x$ is \textbf{totally real} (resp. \textbf{totally positive}) if for all automorphisms $\sigma \in {\rm Aut} (\overline{\bbQ} / \bbQ)$ we have $\sigma (x) \in \bbR$ (resp. $\sigma (x) \in \bbR_{>0}$). 

\subsubsection{} We have:

\begin{lemma}\label{lem Weil}
	 Let $q \in \bbZ_{>0}$ and let $x$ be a totally real algebraic integer such that $$ \max_{\sigma \in {\rm Aut} (\overline{\bbQ} / \bbQ)} |\sigma (x)| \leq 2 q^{1/2}.$$ Then $x$ is a $\Re$-$q$-Weil number.
\end{lemma}

\begin{proof}
	Given $y \in [- 2q^{1/2}, 2q^{1/2}]$, let us denote by $$ \{ \zeta (y) , \zeta (y)^{-1} \} $$ the unique pair of complex-conjugate complex numbers of absolute value $1$ such that $$ y = q^{1/2} \zeta (y) + q^{1/2} \zeta (y)^{-1}.$$ Notice that $q^{1/2} \zeta (x)$ and $q^{1/2} \zeta (x)^{-1}$ are algebraic integers (since they annihilate the polynomial $z^2 -xz + q$) and therefore it is left to see that for every $\sigma \in {\rm Aut} (\overline{\bbQ} / \bbQ)$ we have $| \sigma (\zeta (x)) | = 1$. Notice that, for $y \in [- 2q^{1/2}, 2q^{1/2}]$, the pair $\{ \zeta (y) , \zeta (y)^{-1} \}$ is characterized by having product $1$ and sum $y / q^{1/2}$. Notice that the pair $\{ \sigma (\zeta (x)) , \sigma (\zeta (x))^{-1}\}$ has product $1$ and sum either $\sigma (x) / q^{1/2}$ or $- \sigma (x) / q^{1/2}$, depending on whether $\sigma (q^{1/2})$ is equal to $q^{1/2}$ or to $-q^{1/2}$. Therefore either $$ \{ \sigma (\zeta (x)) , \sigma (\zeta (x))^{-1}\} = \{ \zeta (\sigma (x)) , \zeta (\sigma (x))^{-1} \}$$ or $$ \{ \sigma (\zeta (x)) , \sigma (\zeta (x))^{-1}\} = \{ \zeta (-\sigma (x)) , \zeta (-\sigma (x))^{-1} \}.$$ In any case, we see that $| \sigma (\zeta (x)) | = 1$, as desired.
\end{proof}

\subsubsection{} We have:

\begin{lemma}\label{lem matrix Weil}
	Let $q \in \bbZ_{\ge 1}$ and consider a symmetric matrix $A$, all of whose entries are totally real algebraic integers, and such that for every $\sigma \in {\rm Aut} (\overline{\bbQ} / \bbQ)$ and every eigenvalue $\lambda$ of $\sigma (A)$, we have $|\lambda| \leq 2 q^{1/2}$. Then each eigenvalue of $A$ is a totally real $\Re$-$q$-Weil number.
\end{lemma}

\begin{proof}
	Of course, the eigenvalues of $A$ are algebraic integers. Also, for every $\sigma \in {\rm Aut} (\overline{\bbQ} / \bbQ)$, $\sigma (A)$ is a real symmetrix matrix, and so the eigenvalues of $\sigma (A)$ are real. This implies that the eigenvalues of $A$ are totally real. Finally, by the assumption on the size of eigenvalues, using Lemma \ref{lem Weil}, we obtain that all the eigenvalues of $A$ are $\Re$-$q$-Weil numbers.
\end{proof}

\subsection{Lemmas for Proposition \ref{prop inf dim}}

\subsubsection{}

Given a separable Hilbert space $H$ and a self-adjoint bounded operator $T : H \to H$, let us fix some notations and recollections.

\medskip

Given a closed subspace $H'$ of $H$, we will denote by $P^H_{H'} : H \to H'$ the orthogonal projection and by $I^{H'}_H : H' \to H$ the inclusion. We denote by ${\rm sp} (T) \subset \bbR$ the spectrum of $T$. We will say that a vector $v \in H$ is $T$-cyclic if the span of $\{ T^k v\}_{k \in \bbZ_{\ge 0}}$ is dense in $H$. We denote by ${\rm id}$ the function on $\bbR$, or a closed subset of $\bbR$, sending $x$ to $x$.

\medskip

If $H$ admits a $T$-cylic vector, we have a unique morphism of $C^{\star}$-algebras $\Psi_T : C({\rm sp} (T)) \to \calB (H)$ sending ${\rm id}$ to $T$. Fixing a $T$-cyclic vector $v \in H$, we define a Radon measure $\mu_{T,v}$ on ${\rm sp} (T)$ by sending $f \in C({\rm sp} (T))$ to $\langle \Psi_T (f) v , v \rangle$. The support of $\mu_{T,v}$ is equal to ${\rm sp} (T)$. We have a unique isomorphism of Hilbert spaces $$ \Phi_{T,v} : H \xrightarrow{\sim} L^2 ({\rm sp} (T) , \mu_{T,v} )$$ satisfying the conditions:
\begin{itemize}
	\item $\Phi_{T,v} (T w) = {\rm id} \cdot \Phi_{T,v} (w)$ for all $w \in H$.
	\item $\Phi_{T,v} (v) = 1$.
\end{itemize}

\subsubsection{} We have:

\begin{lemma}\label{lem spectral limit}
	Let $H$ be a separable Hilbert space and let $T : H \to H$ be a bounded self-adjoint operator. Let $$ H_1 \subset H_2 \subset \ldots $$ be a sequence of closed subspaces of $H$ such that the closure of $\cup_{k} H_k$ is equal to $H$. Suppose that there exists in $H$ a $T$-cyclic vector $v$, such that, denoting $T_k := P^H_{H_k} T I^{H_k}_H$ and $v_k := P^H_{H_k} v$, we have that $v_k$ is a $T_k$-cyclic vector in $H_k$ for all $k \in \bbZ_{\ge 1}$. Then $$\iota (\mu_{T,v}) = {\rm lim}_{k \to \infty}^{{\rm weak}} \iota (\mu_{T_k, v_k}),$$ where $\iota (-)$ means the extension by zero of the Radon measures to $\bbR$. We can recover ${\rm sp} (T)$ from the $T_k$'s, as the support of ${\rm lim}_{k \to \infty}^{{\rm weak}} \iota (\mu_{T_k, v_k})$.
\end{lemma}

\begin{proof}
	Unfolding the definitions, we need to check that for every $f \in C(\bbR)$ we have \begin{equation}\label{eq measure conv} \langle \Psi_T (f|_{{\rm sp} (T)}) v , v \rangle = \lim_{k \to \infty} \langle  \Psi_{T_k} (f|_{{\rm sp} (T_k)}) v_k , v_k \rangle.\end{equation} Using the Stone-Weierstrass theorem, it is easy to see that it is enough to check (\ref{eq measure conv}) for $f$ being a polynomial. Denoting by $\calP$ the set of polynomials $f$ for which (\ref{eq measure conv}) holds, clearly $\calP$ is a linear subspace, so it is enough to see that $\calP$ contains ${\rm id}^n$ for all $n \in \bbZ_{\ge 0}$. In other words, given $n \in \bbZ_{\ge 0}$ we need to check that $$ \langle T^n v , v \rangle = {\rm lim}_{k \to \infty} \langle T_k^n v_k , v_k \rangle.$$ This is the same as checking $$ {\rm lim}_{k \to \infty} \langle T^n v , v_k \rangle = {\rm lim}_{k \to \infty} \langle (T P^H_{H_k})^n v , v_k  \rangle $$ and for this it is enough to check that $$ T^n v = \lim_{k \to \infty} (T P^H_{H_k})^n v,$$ so it is enough to see that, as $k \to \infty$, $(T P^H_{H_k})^n$ strongly converges to $T^n$. Notice that if $S_k \to S$ and $S'_k \to S'$ are strongly convergent sequences in $\calB (H)$, and such that $\{ || S_k ||\}_k$ is bounded, then $S_k S'_k$ strongly converges to $S S'$. Therefore, since $T P^H_{H_k}$ strongly converges to $T$ and $\{ || T P^H_{H_k} ||\}_k$ is bounded by $|| T||$, we obtain that $(T P^H_{H_k})^n$ strongly converges to $T^n$.
\end{proof}

\subsubsection{} We have:

\begin{lemma}\label{lem inf dim spectrum}
	Let $a \in \bbR_{>0}$, let $H$ be a Hilbert space with Hilbert basis $\{ e_k \}_{k \in \bbZ_{\ge 0}}$ and let us consider the self-adjoint bounded operator $T : H \to H$ characterized by sending $e_k$ to $a e_{k-1} + a e_{k+1}$ for all $k \in \bbZ_{\ge 0}$, where we set $e_{-1} := 0$. Then ${\rm sp} (T) = [-2a , 2a ]$. We have a unique isomorphism of Hilbert spaces $$ \Phi : H \xrightarrow{\sim} L^2 \left( [-2a,2a] , \ (2 \pi a^2)^{-1} \sqrt{4a^2 - x^2} dx \right)$$ sending $e_0$ to $1$ and satisfying $\Phi (T v) = {\rm id} \cdot \Phi (v)$ for all $v \in H$.
\end{lemma}

\begin{proof}
	Given $n \in \bbZ_{\ge 1}$ let us denote by $H_n \subset H$ the span of $\{ e_k \}_{0 \leq k < n}$ and let us denote $T_n := P^H_{H_n} T I^{H_n}_H : H_n \to H_n$. It is readily seen recursively that $e_0$ is a $T$-cyclic vector in $H$, and also a $T_n$-cyclic vector in $H_n$, for all $n \in \bbZ_{\ge 1}$. Therefore we can use Lemma \ref{lem spectral limit}. We thus should consider the Radon measure on $\bbR$ given by $$ \mu := {\rm lim}^{\rm weak}_{n \to \infty} \iota (\mu_{T_n, e_0}).$$ By Lemma \ref{lem basic matrix} that follows, we have $$ \mu_{T_n, e_0} = \sum_{1 \leq k \leq n} \frac{2(1 - \cos^2 \tfrac{\pi k}{n+1})}{n+1} \cdot \delta_{2a \cdot \cos \tfrac{\pi k}{n+1}}$$ and this tends, as $n \to \infty$, to the measure $$f \mapsto \frac{1}{\pi} \int_0^{\pi} 2 (1-\cos^2 x) f(2a \cdot \cos x) dx = \frac{1}{2 \pi a^2} \int_{-2a}^{2a} \sqrt{4a^2-x^2} f(x) dx.$$ Thus all follows.
\end{proof}

We have used:

\begin{lemma}\label{lem basic matrix}
	Let $n \in \bbZ_{\ge 1}$ and let us consider the $n \times n$ matrix $$ A_n := \begin{pmatrix} 0 & 1/2 & 0 & 0 & \cdots & 0 \\ 1/2 & 0 & 1/2 & 0 & \cdots & 0 \\ 0 & 1/2 & 0 & \ddots & 0 & \vdots \\ 0 & 0 & \ddots & \ddots & 1/2 & 0 \\ \vdots & \vdots & 0 & 1/2 & 0 & 1/2 \\ 0 & 0 & 0 & 0 & 1/2 & 0 \end{pmatrix}.$$ Let us also consider the Chebyshev polynomials of the second kind $U_k (x)$, given recursively by: $$ U_0 (x) := 1, \ U_1 (x) := 2x, \ U_{k} (x) := 2x U_{k-1} (x) - U_{k-2} (x) \ \textnormal{ for } k \in \bbZ_{\ge 2}.$$ One also has \begin{equation}\label{eq Cheb} U_k (\cos x) = \frac{\sin ((k+1) x)}{\sin x}.\end{equation} Then the eigenvalues of $A_n$ are given by $$ \cos \tfrac{\pi k}{n+1}, \ \textnormal{ for } 1 \leq k \leq n$$ and, given an eigenvalue $\lambda$ of $A_n$, an eigenvector of $A_n$ with eigenvalue $\lambda$ is given by $$ v_{\lambda} := \left( U_0 (\lambda) , U_1 (\lambda) , \ldots , U_{n-1} (\lambda ) \right).$$ We have $$ || v_{\lambda} ||^2_2 = \frac{n+1}{2 (1 - \lambda^2 )}.$$
\end{lemma}

\begin{proof}
	All statements are well-known (and easy to see using (\ref{eq Cheb})), except the last equality for the norm $|| v_{\lambda} ||_2$, so let us check it. Given $1 \leq k \leq n$ we want to check that $$ \sum_{j=0}^{n-1} U_j (\cos \tfrac{\pi k}{n+1})^2 = \frac{n+1}{2 \sin^2 \tfrac{\pi k}{n+1}}.$$ But $$ \sum_{j=0}^{n-1} U_j (\cos \tfrac{\pi k}{n+1})^2 = \sum_{j=0}^{n-1} \frac{\sin^2 \tfrac{\pi k (j+1)}{n+1}}{\sin^2 \tfrac{\pi k}{n+1}}$$ and thus we want to see that $$ \sum_{j=0}^{n-1} \sin^2 \tfrac{\pi k (j+1)}{n+1} = \frac{n+1}{2}.$$ This is the same as $$ \sum_{j=0}^{n-1} (1-2\sin^2 \tfrac{\pi k (j+1)}{n+1}) = -1.$$ And indeed we have $$ \sum_{j=0}^{n-1} (1-2\sin^2 \tfrac{\pi k (j+1)}{n+1}) = {\rm Re} \left( \sum_{j=0}^{n-1} e^{\tfrac{2 \pi i k (j+1)}{n+1}} \right) = {\rm Re} \left( \sum_{j=0}^{n} e^{\tfrac{2 \pi i k j}{n+1}} \right) - 1 = -1$$ as desired.
\end{proof}

\end{document}